\documentclass{amsart}
\usepackage[utf8]{inputenc}
\usepackage{amsmath, amssymb, amsthm, amsfonts, mathtools, array, xfrac, float, multirow, indentfirst, url, enumerate, enumitem, comment, afterpage, pifont, makecell, xcolor, hyperref, bbm}
\usepackage[margin=1.5in]{geometry}
\usepackage[backend=bibtex, bibstyle=numeric, citestyle=numeric, sorting=nyt,maxbibnames=99, giveninits=false, isbn=false, url=false, doi=false, eprint=false]{biblatex}
\bibliography{references.bib} 

\usepackage{thmtools}
\usepackage{cleveref}
\allowdisplaybreaks

\newtheorem{theorem}{Theorem}
\newtheorem{lemma}[theorem]{Lemma}

\newtheorem{remark}{Remark}

\renewcommand{\Re}{\mathrm{Re}}

\theoremstyle{definition}
\newcommand{\ov}{\overline}
\newcommand{\mc}{\mathcal}
\newcommand{\mr}{\mathrm}
\newcommand{\md}{\,\mathrm{d}}
\newcommand{\wh}{\widehat}
\newcommand{\wt}{\widetilde}

\newcommand{\legendre}[2]{\ensuremath{\left( \frac{#1}{#2} \right)}}
\newcommand{\con}{\mathrm{con}}
\newcommand{\sinc}{\mathrm{sinc}}
\newcommand{\supp}{\mathrm{supp}\,}

\makeatletter
\declaretheoremstyle
    [headformat={\NOTE}, 
    notebraces={}{}, 
    notefont=\bfseries, 
    preheadhook=\def\thmt@space{}, 
    numbered=no
    ]{namedtheorem}
\makeatother

\begin{document}

\author{Tianyu Zhao}

\address{Department of Mathematics, The Ohio State University, 231 West 18th Ave, Columbus, OH 43210, USA.}

\email{zhao.3709@buckeyemail.osu.edu}

\title{The positivity technique and low-lying zeros of Dirichlet $L$-functions}

\subjclass[2020]{11M06, 11M26, 11M41}
\keywords{Dirichlet $L$-functions, low-lying zeros, explicit formula}

\begin{abstract}
    Assuming the generalized Riemann hypothesis, we sharpen the error terms in a number of estimates related to the distribution of low-lying zeros of Dirichlet $L$-functions. The refinements, building upon earlier work of Omar, come from the use of the positivity technique, which involves choosing certain test functions in the explicit formula. We also improve a result of Hughes and Rudnick on the proportion of Dirichlet $L$-functions modulo a large prime $q$ having small first zeros. In addition, some of our results are generalized to a larger class of $L$-functions, and explicit conditional estimates are provided for the lowest zeros of Dirichlet $L$-functions.
\end{abstract}

\maketitle

\section{Introduction and statement of results}
We investigate several important quantities that naturally arise in the study of low-lying zeros of Dirichlet $L$-functions. Throughout the text $\chi$ denotes a primitive Dirichlet character of modulus $q>1$ unless otherwise specified, and $\rho=\beta+i\gamma$ denotes a non-trivial zero of the associated Dirichlet $L$-function $L(s,\chi)$. Let $|\gamma_\chi|$ (resp., $|\wt{\gamma_\chi}|$) be the height of the lowest (resp., lowest non-real) non-trivial zero of $L(s,\chi)$. That is,
\begin{align*}
    |\gamma_\chi|=&\min\{|\gamma|: L(\beta+i\gamma,\chi)=0, \:0<\beta<1\},\\
    |\wt{\gamma_\chi}|=&\min\{|\gamma|: L(\beta+i\gamma,\chi)=0, \:0<\beta<1, \:\gamma\neq 0\}.
\end{align*}
Further, let $n_\chi=\underset{s=1/2}{\mathrm{ord}} L(s,\chi)$ be the order of vanishing of $L(s,\chi)$ at the central point. 

\subsection{Individual Dirichlet \texorpdfstring{$L$}{}-functions} Using an ingenious complex analysis argument, Siegel \cite[Theorem IV]{Sie} (see also the last paragraph of \cite{Sie}) proved that $|\wt{\gamma_\chi}|\to 0$ as $q\to \infty$, and to be precise, that
\begin{equation}\label{Siegel}
    |\gamma_\chi|\leq \left(\frac{\pi}{2}+o(1)\right)\frac{1}{\log\log\log q}, \qquad
    |\wt{\gamma_\chi}|\leq \left(\pi+o(1)\right)\frac{1}{\log\log\log q}.
\end{equation}
It was recently observed by the author \cite{ZhaoGaps} that an argument of Hall and Hayman \cite{HallHayman} originally formulated to study gaps between zeta zeros can be applied to sharpen the constants in \eqref{Siegel} by a factor of $2$. Omar \cite{Oma1} established the bound $n_\chi<\log q$ for all sufficiently large $q$ and verified Chowla's conjecture, which asserts that $L(\frac{1}{2},\chi)\neq 0$ for all $\chi$ (originally stated only for real $\chi$ \cite{Chow}), for real primitive characters of modulus up to $10^{10}$. Under the generalized Riemann hypothesis (GRH), Omar \cite[Theorem 7]{Oma1} showed that
\begin{equation}\label{Omar n_chi}
    n_\chi\ll \frac{\log q}{\log\log q}.
\end{equation}
This also allowed him \cite[Theorem 11]{Oma1} to conditionally improve \eqref{Siegel} to 
\begin{equation}\label{Omar gamma_chi}
    |\wt{\gamma_\chi}|\ll \frac{1}{\log\log q}.
\end{equation}
Omar's argument is based on the positivity technique, which involves applying a specialized form of Weil's explicit formula to test functions whose Fourier transforms satisfy certain suitable properties, as will be explained in \S\ref{section: preliminaries}. This method was originally developed by Odlyzko and Serre, and subsequently improved by Poitou \cite{Poi}, to furnish lower bounds on discriminants of number fields. We refer the reader to \cite{Odl} for a survey of its historical background and earlier development, and to \cite{Miller, Oma3, Oma2} for further applications to studying the first zeros of various families of $L$-functions.

For an alternative way to arrive at \eqref{Omar n_chi} and \eqref{Omar gamma_chi}, one may start with the Riemann\textendash von Mangoldt formula (see, e.g., \cite[Corollary 14.6]{MV})
\begin{equation}\label{N(T) in terms of S(T)}
    N_0(t,\chi)=\frac{t}{\pi}\log \frac{qt}{2\pi e}+S(t,\chi)+S(t,\ov{\chi})-\frac{\chi(-1)}{4}+O\left(\frac{1}{t+1}\right), \quad t>0. 
\end{equation}
Here $N_0(t,\chi)$ counts the number of zeros of $L(s,\chi)$ with $0<\beta<1$ and $-t\leq \gamma\leq t$ (any zero with ordinate exactly equal to $\pm t$ is counted with weight $1/2$), and
\[
S(t,\chi):=\frac{1}{\pi}\mathrm{arg}\: L\left(\frac{1}{2}+it,\chi\right)=-\frac{1}{\pi}\int_{1/2}^\infty \mr{Im} \frac{L'}{L}(\sigma+it,\chi)\md \sigma
\]
if $t\neq \gamma$ for any $\gamma$, otherwise $S(t,\chi)$ is defined by $\frac{1}{2}(S(t+0,\chi)+S(t-0,\chi))$. Under GRH, it is known that 
\[
S(t,\chi)\ll \frac{\log q(t+1)}{\log\log q(t+3)}
\]
uniformly in $q$ and $t>0$ due to a classical result of Selberg \cite{Sel}. Taking $t\to 0$ on the right-hand side of \eqref{N(T) in terms of S(T)} gives \eqref{Omar n_chi}, and so $N_0(t,\chi)>n_\chi$ when $t\gg (\log\log q)^{-1}$, proving \eqref{Omar gamma_chi}. More recently, using the theory of Beurling\textendash Selberg type extremal functions, Carneiro, Chandee and Milinovich \cite[Theorem 6 \& Example 8]{CCM1} sharpened the conditional bound on $S(t,\chi)$ to 
\begin{equation}\label{CCM S(T,chi)}
    |S(t,\chi)|\leq \frac{1}{4}\frac{\log q(|t|+1)}{\log\log q(|t|+1)}+O\left(\frac{\log q(|t|+1)\log\log\log q(|t|+1)}{(\log\log q(|t|+1))^2}\right).
\end{equation}
Consequently,
\begin{equation}\label{CCM n_chi}
    n_\chi\leq \frac{1}{2}\frac{\log q}{\log\log q}+O\left(\frac{\log q \log\log\log q}{(\log\log q)^2}\right)
\end{equation}
and 
\begin{equation}\label{CCM gamma_chi}
    |\gamma_\chi|\leq \frac{\pi}{2}\frac{1}{\log\log q}+O\left(\frac{\log\log\log q}{(\log\log q)^2}\right).
\end{equation}
By the same reasoning as above, \eqref{CCM n_chi} implies that
\begin{equation}\label{CCM tilde gamma_chi}
    |\wt{\gamma_\chi}|\leq \frac{\pi}{\log\log q}+O\left(\frac{\log\log\log q}{(\log\log q)^2}\right).
\end{equation}

\subsection{The family of Dirichlet \texorpdfstring{$L$}{}-functions modulo \texorpdfstring{$q$}{}}\label{subsection: family mod q}

Much more work has been done towards understanding the distribution of low-lying zeros in families of Dirichlet $L$-functions, such as the family of quadratic Dirichlet $L$-functions and the family of Dirichlet $L$-functions of a fixed modulus. In the latter direction, Murty \cite{Mur} showed under GRH that 
\begin{equation}\label{Murty}
    \frac{1}{\phi(q)}\sum_{\chi\bmod q}n_\chi \leq \frac{1}{2}+O\left(\frac{1}{\log q}\right)
\end{equation}
essentially by using the positivity technique. As a result, the proportion of $\chi$ mod $q$ such that $L(\frac{1}{2},\chi)\neq 0$ is at least $\frac{1}{2}-o(1)$ as $q\to \infty$. By constructing certain extremal functions in reproducing kernel Hilbert spaces, Carneiro, Chirre and Milinovich \cite{CCM2} improved on the average order of vanishing at low-lying heights of the form $\frac{2\pi t}{\log q}$ for $t>0$, but at $t=0$ \eqref{Murty} has not been updated so far \footnote{For the family of primitive Dirichlet $L$-functions with conductors $\in [Q,2Q]$, the non-vanishing proportion at the central point has been improved beyond $\frac{1}{2}$ for large $Q$, unconditionally by Pratt \cite{Pratt} and conditionally by Drappeau \textit{et al.} \cite{DPR}.}. For unconditional results in this direction we refer the reader to \cite{BalMur, IwaSar, Bui, KhaNgo, KMN, QinWu}.

In their seminal work \cite{HuRu}, Hughes and Rudnick established the one-level density for the unitary family of Dirichlet $L$-functions modulo a prime $q$. (The restriction of $q$ to primes simplifies the discussion by ensuring that each non-principal $\chi$ mod $q$ is primitive.) Take the Fourier transform of $f\in L^1(\mathbb{R})$ to be 
\[
\wh{f}(t):=\int_{-\infty}^\infty f(x)e^{-2\pi itx}\md x,
\]
which will also be the convention used throughout this paper. Then, for any real, even test function $f$ with $f(x)\ll_\epsilon (1+|x|)^{-1-\epsilon}$ and $\supp(\wh{f})\subset [-2,2]$, 
\[
\frac{1}{q-2}\sum_{\substack{\chi\bmod q \\ \chi\neq \chi_0}}\sum_{\rho_\chi}f\left(\frac{\log q}{2\pi}\frac{\rho_\chi-1/2}{i} \right)=\int_{-\infty}^\infty f(x)\md x+O\left(\frac{1}{\log q}\right).
\]
As a corollary, they \cite[Corollary 8.2]{HuRu} showed under GRH that
\begin{equation}\label{Hughes Rudnick min gamma_chi}
   \limsup_{\substack{q\to \infty \\ q  \text{\:prime}}}\min_{\substack{\chi\bmod q\\\chi\neq \chi_0}} \frac{|\gamma_\chi| \log q}{2\pi} \leq \frac{1}{4}.
\end{equation}
In other words, there exist zeros whose ordinates are as small as $\frac{1}{4}+o(1)$ of the average spacing between consecutive zeros. Using the variance of the one-level density, they \cite[Theorem 8.3]{HuRu} further proved, also conditionally, that for any $\beta\geq 0.634$,
\begin{multline}\label{Hughes Rudnick proportion}
    \liminf_{\substack{q\to \infty\\q \: \text{prime}}} \frac{1}{q-2}\#\left\{\chi\neq \chi_0:\frac{|\gamma_\chi|\log q}{2\pi}<\beta\right\}\\
    \geq 1-\frac{3+\pi^2+72\beta^2-8\pi^2\beta^2+48\beta^4+16\pi^2\beta^4}{12\pi^2(4\beta^2-1)^2}.
\end{multline}
For instance, we see by taking $\beta=1$ that for large $q$ at least $80\%$ of the characters modulo $q$ have low-lying zeros within the expected height, and this proportion increases to $\frac{11\pi^2-3}{12\pi^2}=0.8913\ldots$ as $\beta\to \infty$. In fact, without assuming any hypothesis, Selberg \cite[Theorem 3]{Sel} demonstrated using a completely different method that this proportion eventually converges to 1 at a rate of at least $1-O(\beta^{-1/2})$, but he did not provide an explicit expression in terms of $\beta$.

\subsection{Statement of results.}

The purpose of this work is further develop Omar's argument in \cite{Oma1} to reprove and refine some of the aforementioned results, in hopes of demonstrating the versatility and efficacy of the positivity technique, which is at the same time relatively simple and straightforward. For individual Dirichlet $L$-functions, we refine \eqref{CCM n_chi}, \eqref{CCM gamma_chi} and \eqref{CCM tilde gamma_chi} by removing a $\log\log\log q$ factor from each error term. In particular, this argument circumvents the intricate analysis of $S(t,\chi)$ and thus has the additional advantage that it can easily be made fully explicit. 

\begin{theorem}\label{theorem gamma_chi}
Assuming GRH, 
\begin{equation}\label{inequality gamma_chi}
    |\gamma_\chi|\leq \frac{\pi}{2\log\log q}+\frac{\pi(\log 4+1)}{2(\log\log q)^2}+O\left(\frac{1}{(\log\log q)^3}\right).
\end{equation}
Moreover, for any constant $C>\log 4+1$,
$L(s,\chi)$ has at least 
\[
\left(\frac{\pi^2}{8}(C-\log 4-1)+o(1)\right)\frac{\log q}{(\log\log q)^2}
\]
non-trivial zeros with height at most 
\[
\frac{\pi}{2\log\log q}+\frac{\pi C}{2(\log\log q)^2}.
\]
\end{theorem}
By making the proof of \eqref{inequality gamma_chi} explicit, we show in \S\ref{section: explicit results} that under GRH,
\[
|\gamma_\chi|\leq \frac{\pi}{2(\log\log q-1.43)}\sqrt{\frac{\log\log q}{\log\log q-2}}
\]
whenever $q\geq 10^{24}$.

Theorem~\ref{theorem gamma_chi} indicates that if \eqref{inequality gamma_chi} turns out to be sharp for some primitive $\chi\bmod q$ (although this is not the case in all likelihood), then for all $C>\log 4+1$, the average spacing between consecutive zeros with height 
\[
|\gamma|\in \left[\frac{\pi}{2\log\log q}+\frac{\pi(\log 4+1)}{2(\log\log q)^2}, \:\frac{\pi}{2\log\log q}+\frac{\pi C}{2(\log\log q)^2}\right]
\]
is at most $(\frac{8}{\pi}+o(1))\frac{1}{\log q}<\frac{2\pi}{\log q}$. 

For the other two quantities we prove the following:

\begin{theorem}\label{theorem n_chi}
Assuming GRH, 
\[
n_\chi\leq \frac{\log q}{2\log\log q}+\frac{(\log 4+1)\log q}{2(\log\log q)^2}+O\left(\frac{\log q}{(\log\log q)^3}\right).
\]
\end{theorem}

\begin{theorem}\label{theorem tilde gamma_chi}
Assuming GRH,
\[
|\wt{\gamma_\chi}|\leq \frac{\pi}{\log\log q}+\frac{\pi(\log 4+1)}{(\log\log q)^2}+O\left(\frac{1}{(\log\log q)^3}\right).
\]
\end{theorem}

We remark that the main terms of these estimates, which are of the same strengths as \eqref{CCM n_chi}--\eqref{CCM tilde gamma_chi}, are best possible for this method barring a breakthrough in detecting cancellations in weighted character sums over primes that appear in the explicit formula (see Remarks~\ref{remark optimality gamma_chi} and \ref{remark optimality n_chi}).

In fact, \eqref{CCM S(T,chi)}--\eqref{CCM tilde gamma_chi} are special cases of results in \cite{CCM1} that are applicable to a larger class of $L$-functions including those associated to cuspidal automorphic representations of $\mathrm{GL}_m$ over a number field. In particular, they sharpen the generalized versions of \eqref{Omar n_chi} and \eqref{Omar gamma_chi} obtained by Omar \cite{Oma2} in this general setting. Thus, for completeness and comparison, we shall give generalizations of Theorems~\ref{theorem gamma_chi}--\ref{theorem tilde gamma_chi} in \S\ref{section: general L function}. 

A natural extension of the method allows us to say something slightly more precise than \eqref{Murty} and \eqref{Hughes Rudnick min gamma_chi} for the family of $\chi\bmod q$. 

\begin{theorem}\label{theorem average n_chi}
Assuming GRH, 
\[
\frac{1}{\phi(q)-1}\sum_{\substack{\chi\bmod q\\\chi\neq \chi_0}} n_\chi\leq \frac{1}{2}-\frac{\log\log q}{2\log q}+O\left(\frac{\log\log\log q}{\log q}\right).
\]
Hence, the proportion of $\chi \bmod q$ satisfying $L(\frac{1}{2},\chi)\neq 0$ is strictly greater than $\frac{1}{2}$ for all sufficiently large $q$.
\end{theorem}

\begin{theorem}\label{theorem min gamma_q}
Assuming GRH,
\[
\min_{\substack{\chi\bmod q\\\chi\neq \chi_0}} \frac{|\gamma_\chi|\log q}{2\pi}\leq \frac{1}{4}-\frac{\log\log q}{4\log q}+O\left(\frac{\log\log\log q}{\log q}\right).
\]
\end{theorem}

With a little extra work we have removed the restriction to prime moduli from \eqref{Hughes Rudnick min gamma_chi} (and from \eqref{Hughes Rudnick proportion}; see Theorem~\ref{theorem: proportion min gamma_q} below). A similar estimate holds for $|\wt{\gamma_\chi}|$ with $1/4$ replaced by $1/2$ on the right-hand side, but the details will be omitted.

In the other direction, we prove the following result on large heights of the first zeros: 
\begin{theorem}\label{theorem: max gamma_chi}
    Assuming GRH, 
    \[
    \max_{\substack{\chi\bmod q\\\chi\neq \chi_0}} \frac{|\gamma_\chi|\log q}{2\pi}\geq \frac{1}{4}+\frac{\gamma_E+\log 8\pi}{4\log q}+O\left(\frac{1}{(\log q)^2}\right)
    \]
    where $\gamma_E=0.5772\ldots$ is Euler's constant.
\end{theorem}

Finally, regarding the proportion of characters with small zeros, we improve Hughes and Rudnick's result \eqref{Hughes Rudnick proportion} for all $1/2<\beta<\infty$. The number $1/2$ here is best possible for this method, but can be further improved to $1/4$ via a different method (see \cite{ZhaoBLMS}).

\begin{theorem}\label{theorem: proportion min gamma_q}
    Let $\alpha_0=1.8652\ldots$ be the unique number in $(1,\infty)$ at which the function $f(\alpha)=\frac{(6\alpha^2+\pi^2-3)(\alpha+1)}{12\pi^2(\alpha-1)}$ is minimized. Assuming GRH, 
    \begin{multline*}
        \liminf_{\substack{q\to \infty}}\frac{1}{\phi(q)-1}\# \left\{\chi\neq \chi_0: \frac{|\gamma_\chi| \log q}{2\pi}<\beta\right\} \\ \geq
        \begin{cases}
        \dfrac{1}{1+\dfrac{3+\pi^2+72\beta^2-8\pi^2\beta^2+48\beta^4+16\pi^2\beta^4}{12\pi^2(4\beta^2-1)^2}}, & 1/2 <\beta<\beta_0,\\
        \dfrac{1}{1+\dfrac{f(\alpha_0)}{4\beta^2}}, & \beta\geq \beta_0
        \end{cases}
    \end{multline*}
    where $\beta_0=\sqrt{\frac{\alpha_0+1}{\alpha_0-1}}/2=0.9098\ldots$ and $f(\alpha_0)=0.7757\ldots$. 
\end{theorem}

In particular, the conditional lower bound we provide grows at a rate of $1-O(\beta^{-2})$, which indeed converges to 1 as $\beta\to \infty$, agreeing with and refining Selberg's unconditional result \cite[Theorem 9]{Sel} mentioned at the end of \S\ref{subsection: family mod q}. A similar estimate for the family of quadratic Dirichlet $L$-functions is obtained in Theorem~\ref{theorem proportion quadratic}. For the proofs, we pursue a cardinality argument based on the Cauchy\textendash Schwarz inequality instead of using Chebyshev's inequality as in \cite{HuRu}, and also allow for more flexibility with the support of the chosen test function as $\beta$ varies (see \S\ref{subsection: proportion} for details).

\begin{figure}[ht]
    \centering
    \includegraphics[width=11cm]{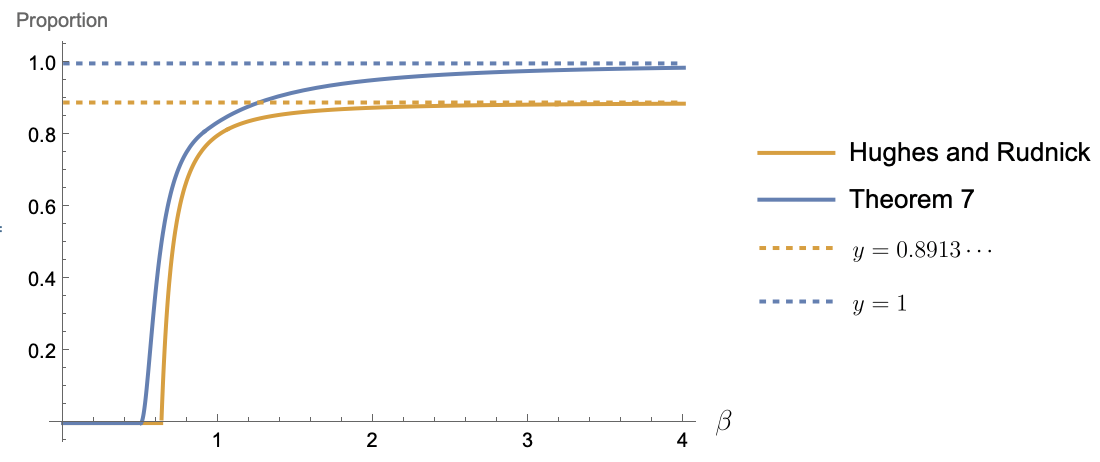}
    \caption{A plot illustrating the lower bounds provided by \eqref{Hughes Rudnick proportion} and Theorem~\ref{theorem: proportion min gamma_q} for small $\beta$.}
    \label{Figure_Proportion_mod_q}
\end{figure}

\section{Preliminaries and lemmas}\label{section: preliminaries}
\subsection{Weil's explicit formula}
We begin by stating the version of Weil's explicit formula formulated by Mestre \cite{Mestre} (see also the work of Barner \cite{Bar}) and adapted for $L(s,\chi)$. We say that a function $F:\mathbb{R}\to \mathbb{R}$ is \textit{admissible} if it is even and satisfies $F(0)=1$ along with the following properties:
\begin{enumerate}[label=(\alph*)]
    \item There exists $\epsilon>0$ such that $F(x)e^{-(1/2+\epsilon)x}$ is of bounded variation, and $F(x)$ is ``normalized" in the sense that $F(x)=\frac{1}{2}(F(x+0)+F(x-0))$ for all $x$.
    \item $(F(x)-1)/x$ is of bounded variation.
\end{enumerate}
Put
\begin{equation}\label{Phi(F)}
    \Phi(F)(s):=\int_{-\infty}^\infty F(x)e^{(s-1/2)x}\md x.
\end{equation}
Then we have the following:
\begin{theorem}[Weil]\label{Weil}
    For any admissible function $F$,
    \begin{equation}\label{explicit formula}
        \sum_\rho \Phi(F)(\rho)=\log \frac{q}{\pi}-I_\chi(F)-J_\chi(F),
    \end{equation}
    where the sum on the left-hand side runs over the non-trivial zeros of $L(s,\chi)$ with $0<\beta<1$ (counted with multiplicity) and
    \begin{align*}
        I_\chi(F)&:=\int_0^\infty \left(\frac{F(x/2)e^{-(1/4+\delta_\chi/2)x}}{1-e^{-x}}-\frac{e^{-x}}{x}\right)\md x, \qquad \delta_\chi=\frac{1-\chi(-1)}{2},\\ J_\chi(F)&:=2\sum_{n=1}^\infty\Re(\chi(n))F(\log n)\frac{\Lambda(n)}{\sqrt{n}}.
    \end{align*}
\end{theorem}
As usual, $\Lambda(n)$ denotes the von Mangoldt function that takes on the value $\log p$ when $n=p^m$ is a prime power and 0 otherwise. 

\subsection{Overview of the positivity technique}
Note that if $\rho=\frac{1}{2}+i\gamma$, then $\Phi(F)(\rho)=\wh{F}(\frac{\gamma}{2\pi})$ and $\Phi(F_T)(\rho)=T\wh F(\frac{T\gamma}{2\pi})$ where $F_T(x):=F(x/T)$. We assume GRH henceforth, so that every non-trivial zero is of this form. 

To see how to bound $|\gamma_\chi|$ from above, suppose that $|\gamma_\chi|\neq 0$ and that $\wh{F}(t)\leq 0$ for $|t|\geq t_0$. Putting $T=2\pi t_0/|\gamma_\chi|$, we have $\Phi(F_T)(\rho)\leq 0$ for each $\rho$, which implies that 
\begin{equation}\label{positivity technique F}
    I_\chi(F_T)+J_\chi(F_T)\geq \log \frac{q}{\pi}
\end{equation}
in view of \eqref{explicit formula}. If we can bound these two quantities from above in terms of $T$ (for practical reasons we also require $F$, and hence $F_T$, to be compactly supported), then for sufficiently large $q$ this will lead to a lower bound on $T$ in terms of $q$, and hence an upper bound on $|\gamma_\chi|$.

The treatment for $|\wt{\gamma_\chi}|$ is similar, but we get a weaker bound because the zeros at $\frac{1}{2}$ make a positive contribution to the sum if $\wh{F}(0)>0$. In order to make this precise we first need to derive an upper bound on $n_\chi$. To this end, we apply \eqref{explicit formula} to $H_T$ where $H$ has non-negative Fourier transform $\wh{H}$ and $T$ is a free parameter. By dropping the contributions from all zeros $\rho\neq \frac{1}{2}$, we see that the left-hand side is at least $n_\chi\wh{H_T}(0)=n_\chi\wh{H}(0)T$, and therefore
\begin{equation}\label{positivity technique H}
n_\chi\wh{H}(0)T\leq \log \frac{q}{\pi}-I_\chi(H_T)-J_\chi(H_T).
\end{equation}
Our task reduces to bounding the right-hand side from above and choosing the optimal $T$ that minimizes $n_\chi$. 

Finally, when considering the family of all non-principal $\chi$ mod $q$ simultaneously, we sum both sides of \eqref{explicit formula} over $\chi$ and exploit the orthogonality of characters instead of trivially bounding each sum $J_\chi$ by the triangle inequality. 

\subsection{Lemmas}
We shall use the following admissible function to derive upper bounds on $|\gamma_\chi|$, for which we require that $\wh{F}(t)\leq 0$ outside of some bounded interval \footnote{See the proof of \cite[Theorem 8.1]{HuRu} for a general way of constructing such functions.}:

\begin{equation}\label{def F^alpha}
    F^\alpha(x)=
    \begin{cases}
        (1-|x|)\cos(\pi x)+\dfrac{\alpha}{\pi}\sin(\pi |x|) & \text{if\:\:$x\in [-1,1]$},\\
        0 & \text{otherwise}
    \end{cases}
\end{equation}
where $\alpha>1$. (Omar fixes $\alpha=3$ in \cite{Oma1}, while our $\alpha$ will be a  parameter depending on $q$.) One can readily verify that 
\begin{equation}
    \wh{F^\alpha}(t)=\left((\alpha+1)-4(\alpha-1)t^2\right)\left(\frac{2\cos(\pi t)}{\pi(1-4t^2)}\right)^2,
\end{equation}
so that $\wh{F^\alpha}(t)\leq 0$ for $|t|\geq \sqrt{\frac{\alpha+1}{\alpha-1}}/2$. We now introduce two lemmas concerning the terms $I_\chi$ and $J_\chi$ that appear in the explicit formula \eqref{explicit formula} when applied to $F^\alpha$. The proofs can easily be adapted for other test functions that will appear later.

\begin{lemma}\label{Lemma I_delta}
    \[
    |I_\chi(F^\alpha_T)|\ll \frac{\alpha}{T}+1,
    \]
    where the implied constant is absolute.
\end{lemma}
\begin{proof}
    Since $F^\alpha_T(x)=F^\alpha(x/T)$ is compactly supported in $[-T,T]$, 
    \begin{multline*}                           
        I_\chi(F^\alpha_T)=\int_0^{2T}\frac{F^\alpha_T(x/2)-1}{1-e^{-x}}e^{-(1/4+\delta_\chi/2)x} \md x+\int_0^\infty \left(\frac{e^{-(1/4+\delta_\chi/2)x}}{1-e^{-x}}-\frac{e^{-x}}{x}\right) \md x\\
        -\int_{2T}^\infty  \frac{e^{-(1/4+\delta_\chi/2)x}}{1-e^{-x}} \md x
    \end{multline*}
    where $\delta_\chi=0$ or $1$ according as $\chi$ is even or odd. Also recall that $F^\alpha_T(0)=1$, so by the mean value theorem there exists $x^*\in (0,1)$ such that the first integral can be rewritten as
    \[
    \frac{1}{2T}F^{\alpha \prime}(x^*)\int_0^{2T} \frac{x e^{-(1/4+\delta_\chi/2)x}}{1-e^{-x}} \md x.
    \]
    Note that $|F^{\alpha \prime}(x)|\ll \alpha$ on $(0,1)$ and that the integral above converges. By Gauss's digamma theorem the second integral equals
    \[
    \frac{\Gamma'}{\Gamma}\left(\frac{1}{4}+\frac{\delta_\chi}{2}\right)=
    \begin{cases}
        \gamma_E+3\log 2+\pi/2 & \text{if $\delta_\chi=0$},\\
        \gamma_E+3\log 2-\pi/2 & \text{if $\delta_\chi=1$}.
    \end{cases}
    \]
    The third integral is plainly $O(1)$, and the proof is therefore complete.
\end{proof}

\begin{lemma}\label{Lemma: sum over p}
For $\alpha\geq 3$ and large $T$,
    \begin{equation}\label{equation sum over p}
        \left|J_\chi(F^\alpha_T)\right| \leq \left(4+O(T^{-1})\right)(\alpha-1)\frac{e^{T/2}}{T}.
    \end{equation}
\end{lemma}

\begin{proof}
    We claim that when $\alpha\geq 3$, 
    \begin{equation}\label{bound on F^alpha}
        F^\alpha(x)\leq (\alpha-1)(1-x) \qquad \text{for} \quad x\in [0,1].
    \end{equation}
    Since $F^{\alpha \prime}(1)=-(\alpha-1)$, it suffices to show that $\frac{\mr{d}^2}{\mr{d}x^2}F^{\alpha}(x)<0$, or 
    \[
    \sin(\pi x)> -\pi(1-x)\cos(\pi x) \quad \text{for} \: x\in (0,1).
    \]
    This is readily seen from the Taylor expansions of both sides.

    As usual, let $\psi(t)=\sum_{n\leq t}\Lambda(n)$ denote the Chebyshev function. Then, using \eqref{bound on F^alpha} we can bound
    \begin{align}\label{prime sum in terms of psi}
        \sum_{n=1}^\infty F^\alpha_T(\log n)\frac{\Lambda(n)}{\sqrt{n}} \leq & (\alpha-1)\int_{1}^{e^T}\frac{\md \psi(t)}{\sqrt{t}}-\frac{\alpha -1}{T}\int_1^{e^T} \frac{\log t}{\sqrt{t}}\md \psi(t)\notag \\
        =& (\alpha-1)\left(\frac{\psi(e^T)}{e^{T/2}}+\frac{1}{2}\int_1^{e^T}\frac{\psi(t)}{t^{3/2}}\md t\right)\notag\\
        &\hspace{1cm}-\frac{\alpha -1}{T}\left(\frac{T}{e^{T/2}}\psi(e^T)-\int_1^{e^T}\psi(t)\left(\frac{1}{t^{3/2}}-\frac{\log t}{2t^{3/2}}\right)\md t\right)\notag\\
        =& \frac{\alpha-1}{T}\int_1^{e^T}\frac{\psi(t)}{t^{3/2}}\md t+\frac{\alpha-1}{2}\int_1^{e^T}\frac{\psi(t)}{t^{3/2}}\left(1-\frac{\log t}{T}\right)\md t\\
        =& (4+O(T^{-1}))(\alpha-1)\frac{e^{T/2}}{T}\notag,
    \end{align}
    where we invoked the prime number theorem (in a weak form) $\psi(t)-t\ll t/\log t$.
\end{proof}

\begin{remark}
     When $\alpha>3$ and $T$ is sufficiently large, we can remove the term $O(T^{-1})$ on the right-hand side of \eqref{equation sum over p} by using a more precise inequality than \eqref{bound on F^alpha} at $x=1$ (available from the Taylor series expansion of $F^\alpha$) and a sharper error term in the prime number theorem. We do not pursue this here since the present estimate suffices for our purposes.
\end{remark}

\section{Estimates for individual Dirichlet \texorpdfstring{$L$}{}-functions: Proofs of Theorems~\ref{theorem gamma_chi}--\ref{theorem tilde gamma_chi}} \label{section: estimates for an individual L(s,chi)}

\subsection{Bounding \texorpdfstring{$|\gamma_\chi|$}{}: Proof of Theorem~\ref{theorem gamma_chi}}
We first show \eqref{inequality gamma_chi}. Assume that $L(\frac{1}{2},\chi)\neq 0$, i.e., $|\gamma_\chi|\neq 0$, since the claim holds trivially otherwise. Consider the function $F^\alpha_T(x)=F^\alpha(x/T)$ as defined in \eqref{def F^alpha}. Here $\alpha$ is a large parameter and $T=\sqrt{\frac{\alpha+1}{\alpha-1}}\pi/|\gamma_\chi|\gg \log\log q$, as can be seen from \eqref{Omar gamma_chi}. Note that $\sum_\rho \Phi(F^\alpha_T)(\rho)=\sum_\gamma T\wh{F^\alpha}(\frac{T\gamma}{2\pi})\leq 0$ because $\wh{F^\alpha}(t)\leq 0$ for all $|t|\geq \sqrt{\frac{\alpha+1}{\alpha-1}}/2$. Thus, in view of \eqref{positivity technique F} and Lemmas~\ref{Lemma I_delta} and ~\ref{Lemma: sum over p}, we have
\begin{align*}
    (4+O(T^{-1}))\frac{\alpha e^{T/2}}{T/2}\geq J_\chi(F^\alpha_T)\geq \log \frac{q}{\pi}-I_\chi(F^\alpha_T)=\log q+O\left(\frac{\alpha}{T}+1\right).
\end{align*}
Set $\alpha=\log\log q$ so that 
\[
\frac{e^{T/2}}{T/2}\geq \Delta:= \frac{\log q+O(\frac{\alpha}{T}+1)}{(4+O(T^{-1}))\alpha} = \frac{\log q+O(1)}{4\log\log q+O(1)}.
\]
Then
\begin{align*}
    \frac{T}{2}\geq &\log \Delta+\log\log \Delta+\log \left(1+\frac{\log\log \Delta}{\log \Delta}\right)\\
    =& \log \Delta+\log\log \Delta+\frac{\log\log \Delta}{\log \Delta}+O\left(\left(\frac{\log\log \Delta}{\log \Delta}\right)^2\right)\\
    = & \log\log q-\log 4+O\left(\frac{1}{\log\log q}\right),
\end{align*}
since
\begin{align*}
    \log\Delta=&\log\log q-\log\log\log q-\log 4+O\left(\frac{1}{\log\log q}\right),\\
\log\log \Delta=&\log\log\log q-\frac{\log\log\log q}{\log\log q}+O\left(\frac{1}{\log\log q}\right).
\end{align*}
After expressing $T$ in terms of $|\gamma_\chi|$ and $\alpha$ and appealing to the fact that $\sqrt{1-x}= 1-x/2+O(x^2)$ as $x\to 0$,
we arrive at
\begin{align*}
    \frac{\pi}{2|\gamma_\chi|} \geq& \left(1-\frac{1}{\log\log q}+O\left(\frac{1}{(\log\log q)^2}\right)\right)\left(\log\log q-\log 4+O\left(\frac{1}{\log\log q}\right)\right)\\
    =& \log\log q-\log 4-1+O\left(\frac{1}{\log\log q}\right).
\end{align*}
This finishes the proof of \eqref{inequality gamma_chi}.

Next we prove the second assertion regarding the number of zeros below $t_1:=\frac{\pi}{2\log\log q}+\frac{\pi C}{2(\log\log q)^2}$ with $C>\log 4+1$. Let $n$ be the positive integer such that $|\gamma_{\chi,n}|< t_1\leq |\gamma_{\chi, n+1}|$ (here $\gamma_{\chi,k}$ denotes the $k$th zero ordinate ordered by height), and our goal is to show that
\begin{equation}\label{nth zero}
    n\geq \left(\frac{\pi^2}{8}(C-\log 4-1)+o(1)\right)\frac{\log q}{(\log\log q)^2}.
\end{equation}
If we take $\alpha=a\log\log q$ for some constant $a>0$ and $T:=\sqrt{\frac{\alpha+1}{\alpha-1}}\pi/ t_1=2(\log\log q-C+1/a+o(1))$, then we have $\Phi(F^\alpha_T)(\rho_k)\leq 0$ for all $k>n$, and another application of \eqref{explicit formula} gives
\begin{align*}
    \log q+O\left(\frac{\alpha}{T}+1\right)-(4+o(1))(\alpha-1)\frac{e^{T/2}}{T/2}\leq \sum_{1\leq k\leq n}\Phi(F^\alpha_T)(\rho_k)\leq nT\wh{F^\alpha}(0)=nT\frac{4(\alpha+1)}{\pi^2}.
\end{align*}
After rearranging the inequality we obtain
\[
n\geq \left(\frac{\pi^2}{4}\frac{1-4a e^{-C+1/a}}{2a}+o(1)\right)\frac{\log q}{(\log\log q)^2}.
\]
Finally, it is a calculus exercise to determine that the optimal choice of $a$ is $(C-\log 4)^{-1}$, which yields \eqref{nth zero}.

\begin{remark}\label{remark optimality gamma_chi}
    By examining the first part of the preceding proof and the proof of Lemma~\ref{Lemma: sum over p}, we see that if the chosen test function $F$ is supported in $[-\tau,\tau]$ and $\wh{F}(t)\leq 0$ for $|t|\geq t_0$, then the coefficient of the leading term $1/\log\log q$ in the bound on $|\gamma_\chi|$ would be $\pi \tau t_0$, which can be improved only if we manage to detect substantial cancellations in the sum $J_\chi(F)$. For our choice $F^\alpha$, $\tau=1$ and $t_0\to 1/2^+$ as $\alpha\to \infty$, leading to the coefficient $\pi/2$. The constant $\pi/2$ is best possible for this method as $\tau t_0>1/2$ for any admissible function $F$. This is a consequence of the work of Carneiro, Ismoilov and Ramos \cite{CIR} as a special case of a more general sign uncertainty problem for bandlimited functions (see Remark (iii) under \cite[Theorem 2]{CIR}). There they showed that $\tau t_0\geq 1/2$ for any relevant nonzero $F\in L^1(\mathbb{R})$ with $F(0)\geq 0$, but the unique extremal function (up to multiplication by a constant) attaining the equality has $F(0)=0$, while we require $F(0)=1$.
\end{remark}

\subsection{Bounding \texorpdfstring{$n_\chi$}{}: Proof of Theorem~\ref{theorem n_chi}}
As in \cite[Lemma 9]{Oma1}, we work with the Fourier pair
\begin{equation}\label{def H}
    H(x)=
    \begin{cases}
        1-|x| & \text{if $|x|<1$},\\
        0 & \text{otherwise}
    \end{cases}
    \qquad \text{and} \qquad \wh{H}(t)=\left(\frac{\sin \pi t}{\pi t}\right)^2.
\end{equation}
As can be seen from the proofs of Lemmas~\ref{Lemma I_delta} and ~\ref{Lemma: sum over p}, we have $I_\chi(H_T)=O(1)$ and 
\[
J_\chi(H_T)\leq 2\sum_{n=1}^\infty H_T(\log n)\frac{\Lambda(n)}{\sqrt{n}}\leq (4+O(T^{-1}))\frac{e^{T/2}}{T/2}.
\]
By \eqref{positivity technique H} and the non-negativity of $\wh{H}(t)$,
\begin{align*}
    n_\chi \wh{H}(0)T\leq \log q+(4+O(T^{-1}))\frac{e^{T/2}}{T/2}+O(1).
\end{align*}
Putting $T=2\log\log q-\Delta$ for some constant $\Delta$, we see that
\[
n_\chi \leq \frac{1}{2}\frac{\log q}{\log\log q}+\left(\frac{\Delta}{4}+\frac{2}{e^{\Delta/2}}\right)\frac{\log q}{(\log\log q)^2}+O\left(\frac{\log q}{(\log\log q)^3}\right),
\]
and the theorem follows on choosing $\Delta=2\log 4$.

\begin{remark}\label{remark optimality n_chi}
    The choice $H(x)$ is optimal. Indeed, suppose that $F$ is admissible with $\supp(F)\subset [-1,1]$ and $\wh{F}$ non-negative, then by the Poisson summation formula
    \[
    \wh{F}(0)\leq \sum_{n\in \mathbb{Z}}\wh{F}(n)= \sum_{n\in \mathbb{Z}}F(n)=F(0)=1.
    \]
\end{remark}

\subsection{Bounding \texorpdfstring{$|\wt{\gamma_\chi}|$}{}: Proof of Theorem~\ref{theorem tilde gamma_chi}}

We work with the test function
\begin{equation}\label{def G}
   G^\alpha(x):=\frac{\pi^2}{2(\alpha+2)}(F^\alpha*H)(x) 
\end{equation}
where $\alpha$ is a large parameter and $(f*g)(x)=\int_{-\infty}^\infty f(y)g(x-y)\md y$ stands for the convolution of $f$ and $g$. Note that $G^\alpha$ is supported in $[-2,2]$ with $G^\alpha(0)=1$. Moreover, $\wh{G^\alpha}(t)=\frac{\pi^2}{2(\alpha+2)} \wh{F^\alpha}(t)\wh{H}(t)\leq 0$ for $|t|\geq \sqrt{\frac{\alpha+1}{\alpha-1}}/2$.

If $|\wt{\gamma_\chi}|\neq 0$, consider $G^\alpha_T(x)$ with $T= \sqrt{\frac{\alpha+1}{\alpha-1}}\pi/|\wt{\gamma_\chi}|$, so that $\Phi(G^\alpha_T)(\rho)=T\wh{G^\alpha}(\frac{T\gamma}{2\pi})\leq 0$ except when $\gamma=0$. Since $F^\alpha(x)\ll \alpha(1-|x|)$ and $H(x)=1-|x|$ for $x\in [-1,1]$, we have $G^\alpha(x)\ll (1-|x|/2)^3$, which gives
\[
J_\chi(G^\alpha_T) \ll \frac{e^{T}}{T^3}
\]
by a similar argument as in the proof of Lemma~\ref{Lemma: sum over p}. Also $I_\chi(G^\alpha_T)=O(1)$ since $G^\alpha_T(x)=O(1)$, and hence
\begin{equation}\label{n_chi T inequality}
    n_\chi\wh{G^\alpha}(0)T\geq \log q+O\left(\frac{e^{T}}{T^3}\right)
\end{equation}
where $\wh{G^\alpha}(0)=\frac{2(\alpha+1)}{\alpha+2}$. A quick proof by contradiction shows that if $aT+b \frac{e^T}{T^3}\geq c$ for $T\geq 3$ and positive numbers $a$, $b$ and $c$, then
\[
T\geq \min\left\{\frac{c}{(1+\Delta) a},\log \frac{\Delta c}{(1+\Delta)b}+3\log\log\frac{\Delta c}{(1+\Delta)b}\right\}
\]
for any $\Delta>0$. In view of \eqref{n_chi T inequality} and Theorem~\ref{theorem n_chi},
\begin{align*}
    T\geq &\min\Bigg\{\frac{(\alpha+2)\log q}{(\alpha+1)(1+\Delta)}\left(\frac{\log q}{\log\log q}+\frac{(\log 4+1)\log q}{(\log\log q)^2}+O\left(\frac{\log q}{(\log\log q)^3}\right)\right)^{-1},\\
    &\hspace{3cm}\log \frac{\Delta \log q}{(1+\Delta)b}+3\log\log\frac{\Delta \log q}{(1+\Delta)b}\Bigg\}
\end{align*}
where $b>0$ is an absolute constant. Choosing $\alpha=\log\log q$ and $\Delta=(\log\log q)^{-2}$, we find that
\begin{align*}
    T\geq &\min\Bigg\{\left(1+\frac{1}{\log\log q}+O\left(\frac{1}{(\log\log q)^2}\right)\right)\left(\log\log q-\log 4-1+O\left(\frac{1}{\log\log q}\right)\right), \\& \hspace{2cm} \log\log q+\log\log\log q+O(1)\Bigg\}\\
    =&\log\log q-\log 4+O\left(\frac{1}{\log\log q}\right),
\end{align*}
which yields the stated estimate for $|\wt{\gamma_\chi}|$ in view of the definition of $T$.

\section{Estimates for the family of Dirichlet \texorpdfstring{$L$}{}-functions modulo \texorpdfstring{$q$}{}: Proofs of Theorems~\ref{theorem average n_chi}--\ref{theorem: proportion min gamma_q}} \label{section: family of L(s,chi)}

\subsection{Lemmas}
We start with an estimate on the average size of conductors of characters modulo $q$.
\begin{lemma}\label{Lemma: aver conductor}
    \[
    \log q-\log\log q+O(1)\leq \frac{1}{\phi(q)-1}\sum_{\substack{\chi\bmod q\\ \chi\neq \chi_0}}\log (\con(\chi))\leq \log q
    \]
    where $\con(\chi)$ stands for the conductor of $\chi$.
\end{lemma}

\begin{proof}
    The upper bound is clear with equality attained whenever $q>2$ is a prime. On the other hand, one can show that (see \cite[Proposition 3.3]{FM})
    \[
    \frac{1}{\phi(q)}\sum_{\chi\bmod q}\log (\con(\chi)) =\log q - \sum_{p|q}\frac{\log p}{p-1}.
    \]
    Let $\omega(q)$ denote the number of distinct prime divisors of $q$, and $p_k$ the $k$th prime. Since $\frac{\log x}{x-1}$ is decreasing, the last sum over $p|q$ is 
    \[
    \leq \sum_{p\leq p_{\omega(q)}}\frac{\log p}{p-1}=\log p_{\omega(q)}+O(1)\leq \log\omega(q)+\log\log(\omega(q)+1)+O(1)\leq \log\log q+O(1)
    \]
    where we used the standard fact that $\omega(q)\ll \frac{\log q}{\log\log q}$.
\end{proof}

Next we introduce an averaged version of Lemma~\ref{Lemma: sum over p}.
\begin{lemma}\label{lemma: sum over p=1 mod q}
    For any fixed $\epsilon>0$,
    \[
    \bigg|\sum_{\substack{\chi\bmod q\\ \chi\neq \chi_0}}J_\chi(F^\alpha_T)\bigg|\ll 
    \begin{cases}
        \alpha e^{T/2}+\alpha \sqrt{q}\log q, & T\leq (1+\epsilon)\log q,\\
        \frac{\alpha  e^{T/2}}{T}, & T>(1+\epsilon)\log q,
    \end{cases}
    \]
where the implied constant depends only on $\epsilon$.
\end{lemma}

\begin{proof}
    By \eqref{bound on F^alpha} and the orthogonality relation of characters
    \begin{equation}\label{ortho}
        \sum_{\substack{\chi\bmod q\\\chi\neq \chi_0}}\chi(n)=
        \begin{cases}
            \phi(q)-1 & \text{if $n\equiv 1 \bmod q$},\\
            0 & \text{if $n\equiv 0\bmod q$},\\
            -1 & \text{otherwise},
        \end{cases}
    \end{equation} 
    we obtain
    \[
    \frac{1}{\phi(q)-1}\sum_{\substack{\chi\bmod q\\ \chi\neq \chi_0}}\sum_{p}\Re(\chi(p))F^\alpha_T(\log p)\frac{\log p}{\sqrt{p}} \leq \alpha\sum_{\substack{p\equiv 1\bmod q \\ p\leq e^T}}\frac{\log p}{\sqrt{p}}\left(1-\frac{\log p}{T}\right).
    \]
    Assume $e^T\geq 2q$, otherwise a trivial estimation suffices. According to the Brun\textendash Titchmarsh theorem \cite[Theorem 2]{MV2}, $\pi(t;q,1)\leq \frac{2t}{\phi(q)\log(t/q)}$ for $t>2q$ where $\pi(t;q,1):=\#\{p: p\equiv 1\bmod q, \:p\leq t\}$, and so the above is
    \begin{align*}
        \leq &\alpha\sum_{\substack{p\equiv 1\bmod q \\ p\leq 2q}}\frac{\log p}{\sqrt{p}}\left(1-\frac{\log p}{T}\right)+\frac{\alpha}{2}\int_{2q}^{e^T}\frac{\pi(t;q,1)(\log t-2)}{t^{3/2}}\left(1-\frac{\log t}{T}\right)\md t\\
        \leq &\frac{\alpha \log (q+1)}{\sqrt{q+1}}+\frac{\alpha}{\phi(q)} \int_{2q}^{e^T}\frac{\log t}{\sqrt{t}\log (t/q)}\left(1-\frac{\log t}{T}\right)\md t.
    \end{align*}
    Estimating the integral and multiplying by $\phi(q)-1$ yields the desired upper bound. For the lower bound, we have
    \begin{align*}
        -\sum_{\substack{\chi\bmod q\\ \chi\neq \chi_0}}\sum_{p}\Re(\chi(p))F^\alpha_T(\log p)\frac{\log p}{\sqrt{p}}
        \leq &\alpha\sum_{\substack{p\not \equiv 0,1\bmod q \\ p\leq e^T}}\frac{\log p}{\sqrt{p}}\left(1-\frac{\log p}{T}\right)\\
        \ll & \alpha \int_2^{e^T} \frac{1}{\sqrt{t}}\left(1-\frac{\log t}{T}\right)\md t\\
        \ll & \frac{\alpha e^{T/2}}{T}.
    \end{align*}
    Finally, the contribution from higher powers of primes can be absorbed into the stated estimate.
\end{proof}

\subsection{Average of \texorpdfstring{$n_\chi$}{}: Proof of Theorem~\ref{theorem average n_chi}}
As in the proof of Theorem~\ref{theorem n_chi}, we apply the explicit formula \eqref{explicit formula} to $H_T$ and now average over all non-principal $\chi\bmod q$. By exploiting the orthogonality relation as in the proof of Lemma~\ref{lemma: sum over p=1 mod q}, we obtain
\begin{align*}
    -\sum_{\substack{\chi\bmod q\\ \chi\neq \chi_0}}J_\chi(H_T)\ll \sum_{p\not\equiv  0, 1\bmod q}H_T(\log p)\frac{\log p}{\sqrt{p}} \ll \frac{e^{T/2}}{T}.
\end{align*}
Since $\con(\chi)\leq q$ for all $\chi$ mod $q$ and $\phi(q)\gg \frac{q}{\log\log q}$, we exploit \eqref{positivity technique H} to write
\begin{align*}
    \frac{1}{\phi(q)-1}\sum_{\substack{\chi\bmod q\\\chi\neq \chi_0}} n_\chi T\leq &\log \frac{q}{\pi}-\frac{1}{\phi(q)-1}\sum_{\substack{\chi\bmod q\\\chi\neq \chi_0}} I_\chi(H_T)+O\left(\frac{e^{T/2}}{\phi(q)T}\right)\\
    =&\log q+O\left(\frac{(\log\log q)e^{T/2}}{qT}+1\right).
\end{align*}
Taking $T=2(\log q+\log\log q-\log\log\log q)$ in the above inequality, we deduce that
\begin{align*}              
    \frac{1}{\phi(q)-1}\sum_{\substack{\chi\bmod q\\\chi\neq \chi_0}} n_\chi 
    \leq& \frac{\log q+O(1)}{2(\log q+\log\log q-\log\log\log q)}\\
    =&\frac{1}{2}-\frac{\log\log q}{2\log q}+O\left(\frac{\log\log\log q}{\log q}\right),
\end{align*}
as desired.

\subsection{Small \texorpdfstring{$|\gamma_\chi|$}{}: Proof of Theorem~\ref{theorem min gamma_q}}

Let $T=\sqrt{\frac{\alpha+1}{\alpha-1}}\pi/|\gamma_q|$ where $|\gamma_q|:=\min_{\chi\neq \chi_0}|\gamma_\chi|$. Assume that $T\geq (1+\epsilon)\log q$. Applying \eqref{positivity technique F} and averaging, we find by Lemmas~\ref{Lemma: aver conductor} and ~\ref{lemma: sum over p=1 mod q} that 
\[
\log q -\log\log q \ll \frac{\alpha e^{T/2}}{\phi(q)T}\ll \frac{\alpha e^{T/2}\log\log q}{qT}.
\]
Now put $\alpha=a\log q$ for some appropriate constant $a>0$ so that $\frac{e^{T/2}}{T/2}\geq \frac{q}{\log\log q}$, which implies that $T/2\geq \log q+\log\log q-\log\log\log q$ for large $q$. Observe that our starting assumption on $T$ for the application of Lemma~\ref{lemma: sum over p=1 mod q} in fact unnecessary, since repeating the above argument under the assumption $T<(1+\epsilon)\log q$ would lead to a contradiction. The lower bound on $T$ in turn gives
\begin{align*}
    \frac{|\gamma_q|\log q}{2\pi}\leq & \frac{\log q}{4(\log q+\log\log q-\log\log\log q)}\sqrt{\frac{a \log q+1}{a\log q-1}}\\
    =&\frac{1}{4}-\frac{\log\log q}{4\log q}+O\left(\frac{\log\log\log q}{\log q}\right).
\end{align*}

\subsection{Large \texorpdfstring{$|\gamma_\chi|$}{}: Proof of Theorem~\ref{theorem: max gamma_chi}}\label{subsection: max gamma_chi}

It will soon be seen that we need an admissible function $K(x)$ supported on $[-1,1]$ with non-negative $\wh{K}$ that makes the quantity
\[
\max \left\{\beta>0: \min_{|t|\leq 2\beta}\wh{K}(t)\geq \frac{1}{2}\right\}
\] 
as large as possible. The following Fourier pair shows that the maximum is at least $1/4$:
\begin{equation}\label{def K}
    \begin{split}
        K(x):=F^1(x)=
        \begin{cases}
            (1-|x|)\cos(\pi x)+\dfrac{1}{\pi}\sin(\pi |x|) & \text{if\:\:$x\in [-1,1]$},\\
            0 & \text{otherwise},
        \end{cases}
    \end{split}
\end{equation}
\[
\wh{K}(t)=\frac{8}{\pi^2}\left(\frac{\cos(\pi t)}{1-4t^2}\right)^2.
\]
We suspect that this test function is optimal. A short calculation reveals that $K(x)$ vanishes up to the second derivative as $x\to 1^-$, and more precisely one has $K(x)\leq 4(1-|x|)^3$. Assuming $T\geq (1+\epsilon)\log q$, it follows as in the proof of Lemma~\ref{lemma: sum over p=1 mod q} that 
\[
-\sum_{\substack{\chi\bmod q\\\chi\neq \chi_0}}J_\chi(K_T) \ll \frac{e^{T/2}}{T^3}.
\]
Moreover, we see from the proof of Lemma~\ref{Lemma I_delta} that 
\[
I_\chi(K_T)=
\begin{cases}
    \gamma_E+3\log 2+\pi/2+O(T^{-1}) & \text{if $\chi(-1)=1$},\\
    \gamma_E+3\log 2-\pi/2+O(T^{-1}) & \text{if $\chi(-1)=-1$},
\end{cases}
\]
and thus
\[
-\frac{1}{\phi(q)-1}\sum_{\substack{\chi\bmod q\\\chi\neq \chi_0}} I_\chi(K_T)=-\gamma_E-3\log 2+O(T^{-1})+O(\phi(q)^{-1})
\]
by the orthogonality relation \eqref{ortho}.

For $\beta>0$, denote by $N(\frac{2\pi \beta}{\log q},\chi)$ the number of zeros of $L(s,\chi)$ with $|\gamma|\leq \frac{2\pi \beta}{\log q}$. The non-negativity of $\wh{K}$ implies that
\begin{align*}
    \frac{1}{\phi(q)-1}\sum_{\substack{\chi\bmod q\\\chi\neq \chi_0}}  N\left(\frac{2\pi \beta}{\log q},\chi\right) &T \min_{|t|
    \leq \frac{\beta T}{\log q}} \wh{K}(t)\\
    \leq & \frac{1}{\phi(q)-1}\sum_{\substack{\chi\bmod q\\\chi\neq \chi_0}} \sum_\gamma T\wh{K}\left(\frac{T\gamma}{2\pi}\right)\\
    \leq & \log \frac{q}{\pi}-\gamma_E-3\log 2+O\left(\frac{(\log\log q)e^{T/2}}{qT^3}+\frac{1}{T}\right).
\end{align*}
Choose $T=2\log q$, so the big-$O$ term in the previous line is $O((\log q)^{-1})$. We aim to maximize $\beta$ subject to the condition that
\begin{equation}\label{aver. no. of zeros < 1}
    \log \frac{q}{\pi}-\gamma_E-3\log 2+O\left(\frac{1}{\log q}\right)<2\log q \cdot \min_{|t|
    \leq 2\beta} \wh{K}(t)
\end{equation}
for all large $q$, since then the average number of zeros below $\frac{2\pi \beta}{\log q}$ will be strictly less than 1. As $\wh{K}(t)$ is monotonically decreasing on $|t|\leq 2\beta$ (for $\beta<1/2$, say), the condition \eqref{aver. no. of zeros < 1} can be rewritten as
\[
\frac{8}{\pi^2}\left(\frac{\cos(2\pi \beta)}{1-16\beta^2}\right)^2 > \frac{1}{2}\left(1-\frac{\log \pi+\gamma_E+3\log 2}{\log q}+O\left(\frac{1}{(\log q)^2}\right)\right).
\]
If we take $\beta=(1+\tilde{\beta})/4$ where $\tilde{\beta}=o(1)$, then the left-hand side becomes
\[
\frac{8}{\pi^2}\left(\frac{\sin(\pi \tilde{\beta}/2)}{2\tilde{\beta}+ \tilde{\beta}^2}\right)^2=\frac{1}{2}\left(1-\tilde{\beta}+O(\tilde{\beta}^2)\right),
\]
and hence we can choose $\tilde{\beta}$ to be as large as
\[
\tilde{\beta}=\frac{\log \pi+\gamma_E+3\log 2}{\log q}+O\left(\frac{1}{(\log q)^2}\right),
\]
as required.

\begin{remark}
    We briefly discuss a related proportion result. For $\theta\in [0,\pi]$, consider the class of admissible functions
    \begin{equation}\label{def L theta}
        L^\theta(x):=
        \begin{cases}
        (1-|x|)\dfrac{\sinc(\theta(1-|x|))+\cos(\theta x)}{\sinc(\theta) +1} & \text{if \:\: $x\in [-1,1]$},\\
        0 & \text{otherwise}
        \end{cases}
    \end{equation}
    \footnote{Here, as usual, $\sinc(x)$ is defined by $\frac{\sin x}{x}$ if $x\neq 0$ and $1$ otherwise.} with non-negative Fourier transform
    \[
    \wh{L^\theta}(t)=\dfrac{\left(\sinc(\theta/2-\pi t)+\sinc(\theta/2+\pi t)\right)^2}{2\, \sinc(\theta)+2}.
    \]
    Note that $L^0(x)=H(x)$ and $L^\pi(x)=K(x)$ with $H$ as in \eqref{def H} and $K$ as in \eqref{def K}. From the argument in \S\ref{subsection: max gamma_chi} we observe that for $0\leq \beta\leq 1/4$, 
    \begin{align*}
        \liminf_{\substack{q\to \infty}}\frac{1}{\phi(q)-1}\# \left\{\chi\neq \chi_0: \frac{|\gamma_\chi| \log q}{2\pi}>\beta\right\} \geq & 1-\limsup_{q\to \infty}\frac{1}{\phi(q)-1}\sum_{\substack{\chi\bmod q\\\chi\neq \chi_0}}  N\left(\frac{2\pi \beta}{\log q},\chi\right)\\   
        \geq & 1-\frac{1}{2\cdot \max_{\theta\in [0,\pi]}\wh{L^\theta}(2 \beta)} \\
        =& 1-\frac{1}{1+\sinc(4\pi\beta)}.
    \end{align*}
    The cases $\beta=0$ and $1/4$ correspond to Theorems~\ref{theorem average n_chi} and ~\ref{theorem: max gamma_chi}, respectively.
\end{remark}

\subsection{Proportion of \texorpdfstring{$\chi\bmod q$}{} having small zeros: Proof of Theorem~\ref{theorem: proportion min gamma_q}}\label{subsection: proportion}

Hughes and Rudnick's proof of \eqref{Hughes Rudnick proportion} relies on the following result:
\begin{lemma}\label{HuRu proportion lemma}
    Let $f(x)$ be a real, even function with $f(x)\ll (1+|x|)^{-1-\epsilon}$ such that $\wh{f}$ is supported in $[-1,1]$ and that $\wh{f}(0)=1$. If $f(x)\leq 0$ for $|x|\geq \beta$ with $\beta>1/2$, then 
    \begin{equation}\label{HuRu proportion inequality}
        \liminf_{\substack{q\to \infty\\ q\:\text{prime}}}\frac{1}{q-2}\# \left\{\chi\neq \chi_0: \frac{|\gamma_\chi| \log q}{2\pi}<\beta\right\}\geq 1-\sigma(f)^2
    \end{equation}
    where $\sigma(f)^2:=\int_{-1}^{1}|t| \wh{f}(t)^2\md t$.
\end{lemma}

\begin{proof}
    See the proof of \cite[Theorem 8.3]{HuRu}. 
\end{proof}

The test function $f$ they chose to work with is $\wh{F^\alpha}$ where $\alpha$ is appropriately determined in terms of $\beta$ so that $\wh{F^\alpha}(x)\leq 0$ for $|x|\geq \beta$. We now mention two ways of improving \eqref{Hughes Rudnick proportion}. First note that in order for this method to produce a non-vacuous result we need the right-hand side of \eqref{HuRu proportion inequality} to be positive, which only occurs when $\beta\geq 0.6332\ldots$ for $f=\wh{F^\alpha}$, and this can be avoided if one uses Cantelli's inequality (also referred to as the one-sided Chebyshev's inequality) in place of Chebyshev's inequality in the proof of \cite[Theorem 8.3]{HuRu}, which always yields a non-trivial proportion for $\beta>1/2$. Second of all, since the support of $\wh{f}$ is only required to be contained in rather than strictly equal to $[-1,1]$, we have the flexibility of stretching a fixed $f$ (or shrinking $\wh{f}$) horizontally as $\beta$ increases, which keeps $\wh{f}(0)$ invariant while decreasing $\sigma(f)^2$. 

We now prove Theorem~\ref{theorem: proportion min gamma_q}. To make our proof self-contained and consistent with the previous presentation in this paper, we shall again start from the explicit formula instead of assuming results from \cite{HuRu}. Let $\mc{Q}_\beta=\{\chi \bmod q: \frac{ |\gamma_\chi|\log q}{2\pi}< \beta\}$ and  
\begin{equation}\label{def T proportion}
    T=\sqrt{\frac{\alpha+1}{\alpha-1}}\frac{\pi}{2\pi \beta/\log q}=\sqrt{\frac{\alpha+1}{\alpha-1}}\frac{\log q}{2\beta}.
\end{equation}
To be able to control the square of the prime sum we impose the additional assumption that $T\leq \log q$, which holds for all large enough $\alpha$ so long as $\beta>1/2$ (our choice of $\alpha$ will be a constant depending only on $\beta$, not on $q$, and the big-$O$ terms in what follows may depend on $\alpha$). By the Cauchy\textendash Schwarz inequality
\begin{equation}\label{Q inequality}
    \# \mc{Q}_\beta\geq \frac{\left(\sum_{\chi\in \mc{Q}_\beta}\sum_{\rho_\chi}\Phi(F^\alpha_T)(\rho_\chi)\right)^2}{\sum_{\chi\in \mc{Q}_\beta}\left(\sum_{\rho_\chi}\Phi(F^\alpha_T)(\rho_\chi)\right)^2}.
\end{equation}
We first bound the numerator from below. The key observation is again that
\[
\sum_{\substack{\chi \not \in \mc{Q}_\beta\\\chi\neq \chi_0}}\sum_{\rho_\chi}\Phi(F^\alpha_T)(\rho_\chi)\leq 0,
\]
which implies that 
\begin{align*}
    \sum_{\chi\in \mc{Q}_\beta}\sum_{\rho_\chi}\Phi(F^\alpha_T)(\rho_\chi)
    \geq & \sum_{\chi\neq \chi_0}\sum_{\rho_\chi}\Phi(F^\alpha_T)(\rho_\chi)\\
    \geq & (\phi(q)-1)(\log q-\log\log q+O(1))-\sum_{\chi\neq \chi_0} I_\chi(F^\alpha_T)-\sum_{\chi\neq \chi_0} J_\chi(F^\alpha_T)\\
    =& \phi(q)(\log q-\log\log q+O(1)),
\end{align*}
where we applied Lemmas~\ref{Lemma: aver conductor} and ~\ref{lemma: sum over p=1 mod q}. On the other hand, the denominator is at most 
\begin{multline}\label{second moment zero sum}
    \sum_{\chi\neq \chi_0}\left(\sum_{\rho_\chi}\Phi(F^\alpha_T)(\rho_\chi)\right)^2\\
    \leq \phi(q)(\log q+O(1))^2-2\sum_{\chi\neq \chi_0}(\log(\con(\chi))+O(1))J_\chi(F^\alpha_T)+4\sum_{\chi\neq \chi_0}J_\chi(F^\alpha_T)^2.
\end{multline}
The second term is, by the same lemmas and the Cauchy-Schwarz inequality,
\begin{align*}
    \ll& \sqrt{q}\log q+\sqrt{\sum_{\chi\neq \chi_0} (\log q-\log(\con(\chi))+O(1))^2\sum_{\chi\neq \chi_0}J_\chi(F^\alpha_T)^2}\\
    \ll& \sqrt{q}\log q+\sqrt{\phi(q)\log q \log\log q}\sqrt{\sum_{\chi\neq \chi_0}J_\chi(F^\alpha_T)^2}.
\end{align*}
According to \eqref{ortho},
\begin{align*}
    \sum_{\chi\neq \chi_0}J_\chi(F^\alpha_T)^2=& 4\sum_{\chi\neq \chi_0}\sum_{n_1,n_2}\Re(\chi(n_1))\Re(\chi(n_2)) F^\alpha_T(\log n_1)F^\alpha_T(\log n_2)\frac{\Lambda(n_1)\Lambda(n_2)}{\sqrt{n_1n_2}} \\
    \leq & \sum_{\chi\neq \chi_0}\sum_{n_1,n_2}\left(\chi(n_1 n_2)+\overline{\chi(n_1 n_2)}+\overline{\chi(n_1)}\chi( n_2)+\chi(n_1)\overline{\chi( n_2)}\right)\\
    &\hspace{2cm} \times F^\alpha_T(\log n_1)F^\alpha_T(\log n_2)\frac{\Lambda(n_1)\Lambda(n_2)}{\sqrt{n_1n_2}}\\
    \leq & 2\phi(q)\sum_{n_1n_2\equiv 1\bmod q}F^\alpha_T(\log n_1)F^\alpha_T(\log n_2)\frac{\Lambda(n_1)\Lambda(n_2)}{\sqrt{n_1n_2}}\\
    &\hspace{2cm}+2\phi(q)\sum_{n_1\equiv n_2\bmod q} F^\alpha_T(\log n_1)F^\alpha_T(\log n_2)\frac{\Lambda(n_1)\Lambda(n_2)}{\sqrt{n_1n_2}}.
\end{align*}
Label the last two terms by $J_1$ and $J_2$. First,
\begin{align*}
     J_1 \ll & 2\phi(q) \sum_{k\leq \frac{e^{2T}-1}{q}}\sum_{\substack{p_1p_2=kq+1\\1\leq p_1, p_2\leq e^T}}F^\alpha_T(\log p_1)F^\alpha_T(\log p_2)\frac{\log p_1 \log p_2}{\sqrt{kq+1}}\\
    \ll& 
    \frac{\phi(q)}{\sqrt{q}} \sum_{k\leq \frac{e^{2T}}{q}}\frac{(\log kq)^2}{\sqrt{k}}\max_{\substack{\log a+\log b=\log (kq+1)\\0\leq \log a, \log b\leq T}}\left(1-\frac{\log a}{T}\right)\left(1-\frac{\log b}{T}\right)\\
    \ll & \frac{\phi(q)}{\sqrt{q}} \sum_{k\leq \frac{e^{2T}}{q}}\frac{(\log kq)^2\left(1-\frac{\log kq}{2T}\right)^2}{\sqrt{k}}
    \ll e^T \frac{\phi(q)}{q}
    \leq  \phi(q),
\end{align*}   
where we used \eqref{bound on F^alpha} in the second step. Next observe that if $n_1\equiv n_2\bmod q$ and $n_1,n_2\leq e^T\leq q$, then we must have $n_1=n_2$. Therefore,
\begin{align*}
    J_2=&2\phi(q)\sum_{n_1\leq e^T}F^\alpha_T(\log n_1)^2\frac{\Lambda(n_1)^2}{n_1}\\       
    =&2\phi(q)\sum_{p\leq e^{T}}F^\alpha_T(\log p)^2\frac{(\log p)^2}{p}+2\phi(q)\sum_{\substack{p^k\leq e^T\\ k\geq 2}}F^\alpha_T(k\log p)^2\frac{(\log p)^2}{p^k}\\
    =&2\phi(q)\int_2^{e^T} F^\alpha_T(\log x)^2\frac{\log x}{x}\md x+O(\phi(q))\\
    =&2\phi(q)T^2\int_{\frac{\log 2}{T}}^1 u F^\alpha(u)^2 \md u +O(\phi(q)).
\end{align*}
Recall that $\sigma(F^\alpha)^2=2\int_0^1 u F^\alpha(u)^2\md u$. Inserting the previous estimates into \eqref{Q inequality} gives
\[
\# \mc{Q}_\beta\geq \dfrac{\phi(q)^2(\log q-\log\log q+O(1))^2}
{\phi(q)(\log q+O(1))^2+\phi(q)T^2\sigma(F^\alpha)^2+O(\phi(q)+\phi(q)\sqrt{\log q\log\log q}T\sigma(F^\alpha))}.
\]
Recalling the definition \eqref{def T proportion} of $T$, we need $\sqrt{\frac{\alpha+1}{\alpha-1}}\leq 2\beta$ for $T\leq \log q$, and hence 
\begin{align*}
    \liminf_{\substack{q\to \infty}}\frac{\# \mc{Q}_\beta}{\phi(q)-1}\geq &\frac{1}{1+\frac{1}{4\beta^2}\min_{\sqrt{\frac{\alpha+1}{\alpha-1}}\leq 2\beta} \frac{\alpha+1}{\alpha-1}\sigma(F^\alpha)^2} \\
    =&\frac{1}{1+\frac{1}{4\beta^2}\min_{\alpha\geq \frac{4\beta^2+1}{4\beta^2-1}}\frac{\alpha+1}{\alpha-1}\frac{6\alpha^2+\pi^2-3}{12\pi^2}}.
\end{align*}
The function $f(\alpha)=\frac{\alpha+1}{\alpha-1}\frac{6\alpha^2+\pi^2-3}{12\pi^2}$ attains its minimum on $(1,\infty)$ at $\alpha_0=1.8652\ldots$ with $f(\alpha_0)=0.7757\ldots$, and is increasing on $[\alpha_0,\infty)$. It thus follows that
\[
\liminf_{\substack{q\to \infty}}\frac{\# \mc{Q}_\beta}{\phi(q)-1} \geq
\begin{cases}
    \dfrac{1}{1+\frac{1}{4\beta^2}f(\frac{4\beta^2+1}{4\beta^2-1})} & \text{if}\:\:1/2<\beta< \sqrt{\frac{\alpha_0+1}{\alpha_0-1}}/2,\\
    \dfrac{1}{1+\frac{f(\alpha_0)}{4\beta^2}}& \text{if}\:\:\beta\geq \sqrt{\frac{\alpha_0+1}{\alpha_0-1}}/2,
\end{cases}
\]
which concludes the proof.

\begin{remark}\label{remark: higher moments}
    Higher moments would not help us in this proof. One might attempt to use H\"{o}lder's inequality in place of \eqref{Q inequality} to write
    \[
    \# \mc{Q}_\beta\geq \frac{\left(\sum_{\chi\neq \chi_0}\sum_{\rho_\chi}\Phi(F^\alpha_T)(\rho_\chi)\right)^\frac{r}{r-1}}{\left(\sum_{\chi\neq \chi_0}\left(\sum_{\rho_\chi}\Phi(F^\alpha_T)(\rho_\chi)\right)^r\right)^{\frac{1}{r-1}}}
    \]
    where $r\in \mathbb{Z}_{\geq 2}$. The $r$-th moment in the denominator can be evaluated as in \cite[Theorem 5.1]{RS}. First note that this would force us to take $T\leq \frac{2\log q}{r}$, or $\sqrt{\frac{\alpha+1}{\alpha-1}}<4\beta/r$, so we need a larger threshold $\beta>r/4$ to obtain a positive proportion if $r>2$. In fact, we claim that the choice $r=2$ is optimal for all $\beta$. To this end, we compute (for simplicity take $q$ to be prime) that for all admissible $\alpha$,
    \begin{align*}
        &\lim_{q\to \infty} \frac{1}{\phi(q)(\log q)^r}\sum_{\chi\neq \chi_0}\left(\sum_{\rho_\chi}\Phi(F^\alpha_T)(\rho_\chi)\right)^r \\
        =& \lim_{q\to \infty} \frac{1}{\phi(q)(\log q)^r}\sum_{\chi\neq \chi_0}\left(\log q+O(1)+J_\chi(F^\alpha_T)\right)^r\\
        =& \lim_{q\to \infty} \frac{1}{\phi(q)(\log q)^r}\sum_{m=0}^r \binom{r}{m} (\log q)^{r-m} \sum_{\chi\neq \chi_0} J_\chi(F^\alpha_T)^{m}\\
        =& \lim_{q\to \infty} \frac{1}{\phi(q)(\log q)^r} \sum_{\substack{m=0\\m\:\text{even}}}^r \binom{r}{m} (\log q)^{r-m} \frac{m!}{2^{m/2}(m/2)!}T^m\sigma(F^\alpha)^m\\
        =&\sum_{\substack{m=0\\m\:\text{even}}}^r  \binom{r}{m} \frac{m!}{2^{m/2}(m/2)!}\frac{f(\alpha)^{m/2}}{(2\beta)^m}.
    \end{align*}
    We thus arrive at
    \begin{align*}
        \liminf_{q\to \infty} \frac{\# \mc{Q}_\beta}{\phi(q)-1}\geq & \dfrac{1}{\inf_{\sqrt{\frac{\alpha+1}{\alpha-1}}<\frac{4\beta}{r}} \left(1+\displaystyle \sum_{k=1}^{\lfloor r/2\rfloor} \binom{r}{2k} \frac{(2k)!}{2^k k!}\left(\frac{f(\alpha)}{4\beta^2}\right)^k\right)^{\frac{1}{r-1}}}
    \end{align*}
    when $\beta>r/4$. To see that this proportion is largest when $r=2$, we may assume that $\beta>3/4$. In this range it is easy to verify that $g(\beta):=\inf_{\sqrt{\frac{\alpha+1}{\alpha-1}}<2\beta} \frac{f(\alpha)}{4\beta^2}<1/2$, and it follows that for all $r>2$ 
    \begin{align*}
        \inf_{\sqrt{\frac{\alpha+1}{\alpha-1}}<\frac{4\beta}{r}} &\left\{1+\displaystyle \sum_{k=1}^{\lfloor r/2\rfloor} \binom{r}{2k} \frac{(2k)!}{2^k k!}\left(\frac{f(\alpha)}{4\beta^2}\right)^k \right\} \\  
        \geq &1+\displaystyle \sum_{k=1}^{\lfloor r/2\rfloor} \binom{r}{2k} \frac{(2k)!}{2^k k!}g(\beta)^k\\
        \geq & 1+\binom{r}{2}g(\beta)+\sum_{k=2}^{\lfloor r/2\rfloor} \binom{r}{2k} \frac{(2k)!}{2^k k!}\left(g(\beta)^{2k-1}+g(\beta)^{2k}\right)\\ 
        > & 1+ \binom{r-1}{1} g(\beta)+\binom{r-1}{2} g(\beta)^2+\sum_{k=2}^{\lfloor r/2\rfloor} \binom{r}{2k} \frac{(2k)!}{2^k k!}\left(g(\beta)^{2k-1}+g(\beta)^{2k}\right)\\
        \geq & 1+\sum_{j=1}^{r-1} \binom{r-1}{j} g(\beta)^j\\
        =& (1+g(\beta))^{r-1}.
    \end{align*}
\end{remark}

\section{An analogue of Theorem~\ref{theorem: proportion min gamma_q} for quadratic Dirichlet \texorpdfstring{$L$}{}-functions}

The low-lying zeros of the family of real Dirichlet $L$-functions have also been extensively studied. Given a fundamental discriminant $d$, write $\chi_d(\cdot):=\legendre{d}{\cdot}$ for the associated Kronecker symbol. By computing the asymptotic formulas for the first and second moments of the central values $L(\frac{1}{2},\chi_d)$, Jutila \cite{JutMeanValue} showed that $L(\frac{1}{2},\chi_d)\neq 0$ for at least $X/\log X$ of the fundamental discriminants $0<d\leq X$. Later Soundararajan \cite{Sound} worked with mollified moments and managed to prove that $L(\frac{1}{2},\chi_d)\neq 0$ for $\frac{7}{8}$ of the $d$'s. At around the same time, \"{O}zl\"{u}k and Snyder \cite{OS99} proved a slightly larger non-vanishing proportion $\frac{15}{16}$ assuming GRH by establishing the one-level density for the symplectic family of quadratic Dirichlet $L$-functions for test functions $f$ with $\supp(\wh{f})\in [-2,2]$, which was independently discovered by Katz and Sarnak \cite{KS}. 

We now modify the argument in \S\ref{subsection: proportion} and apply it to quadratic Dirichlet $L$-functions. We start with a few lemmas. The first is a mean square estimate on real character sums due to Jutila. Throughout let $D(X)=\{d: X\leq |d|\leq 2X\}$ and $X^*=|D(X)|\sim \frac{6}{\pi^2}X$.
\begin{lemma}
    \begin{equation}\label{Jutila}
        \sum_{\substack{n\leq N\\ n\not \neq \square}}\Big|\sum_{d\in D(X)}\chi_d(n)\Big|^2 \ll N X (\log N)^{10},
    \end{equation}
    where the outer sum runs over non-square positive integers $n\leq N$.
\end{lemma}
\begin{proof}
    See \cite[Corollary to Theorem 1]{Jut}.
\end{proof}

From now on the big-$O$ terms may depend on $\alpha$.
\begin{lemma}
    \begin{equation}\label{prime sum 1-level}
        2\sum_{d\in D(X)}\sum_{n=1}^\infty\chi_d(n)F^\alpha_T(\log n)\frac{\Lambda(n)}{\sqrt{n}}=X^*T\wh{F^\alpha}(0)/2+O(\sqrt{X}e^{T/2}T^6+X).
    \end{equation}
\end{lemma}
\begin{proof}
    The sum in question can be split as
    \[
    2\sum_{p\leq e^T} \sum_{d\in D(X)}\chi_d(p) F^\alpha_T(\log p)\frac{\log p}{\sqrt{p}}+2\sum_{p\leq e^{T/2}} \sum_{d\in D(X)} \chi_d(p^2)F^\alpha_T(2\log p)\frac{\log p}{p}+O(X^*).
    \]
    By the Cauchy\textendash Schwarz inequality and \eqref{Jutila}, the sum over primes contributes 
    \[
    \ll \sqrt{\sum_{p\leq e^T}\Big|\sum_{d\in D(X)} \chi_d(p)\Big|^2} \sqrt{\sum_{p\leq e^T}\left(F^\alpha_T(\log p)\frac{\log p}{\sqrt{p}}\right)^2}\ll \sqrt{X}e^{T/2}T^6,
    \]
    while the sum over squares of primes equals 
    \begin{align*}
        2X^*\sum_{p\leq e^{T/2}} F^\alpha_T(2\log p)\frac{\log p}{p}+O\left(\sum_{p\leq e^{T/2}} \frac{X}{p} F^\alpha_T(2\log p)\frac{\log p}{p}\right) = X^*T\int_0^1 F^\alpha(x)\md x+O(X).
    \end{align*}
\end{proof}

\begin{lemma}\label{lemma prime sum 2-level}
    For $T\leq \log X-15\log\log X$,
    \[
    \sum_{d\in D(X)} \prod_{i=1}^2 \sum_{n_i}\chi_d(n_i) F^\alpha_T(\log n_i)\frac{\Lambda(n_i)}{\sqrt{n_i}}=X^* T^2\int_0^1 u F^\alpha(u)^2\md u+O\left(\frac{XT^2}{\log\log X}\right).
    \]
\end{lemma}
\begin{proof}
    This follows from the work of Gao \cite[\S 3.6]{Gao} on the $n$-level density, specialized to the case $n=2$. Roughly speaking, only the contribution from the diagonal terms $p_1=p_2$ matters. 
\end{proof}

Put $\mc{D}_\beta=\{d\in D(X): \frac{|\gamma_{\chi_d}|\log X}{2\pi}<\beta\}$. Further let $T=\sqrt{\frac{\alpha+1}{\alpha-1}}\frac{\log X}{2\beta}$ where $\alpha>1$ is chosen such that $\sqrt{\frac{\alpha+1}{\alpha-1}}<2\beta$. We then proceed with a Cauchy\textendash Schwarz argument as before. On the one hand, by the positivity trick and \eqref{prime sum 1-level} we have
\begin{align*}
    \sum_{d\in \mc{D}_\beta}\sum_{\rho_d} \Phi(F^\alpha_T)(\rho_d)\geq & X^* \log X-2\sum_{d\in D(X)}\sum_{n=1}^\infty\chi_d(n)F^\alpha_T(\log n)\frac{\Lambda(n)}{\sqrt{n}}+O(X)\\
    = & X^*\log X-X^*T\wh{F^\alpha}(0)/2+O(X)\\
    \geq & X^*\log X \left(1-\frac{\sqrt{\frac{\alpha+1}{\alpha-1}}(\alpha+1)}{\pi^2\beta}+o(1)\right).
\end{align*}
Note, however, that the right-hand side needs to be non-negative in order to keep the inequality valid upon squaring both sides. That is, we need $\sqrt{\frac{\alpha+1}{\alpha-1}}(\alpha+1)<\pi^2\beta$. On the other hand, 
\begin{align*}
    \sum_{d\in \mc{D}_\beta}&\left(\sum_{\rho_d}\Phi(F^\alpha_T)(\rho_d)\right)^2\\
    &=X^*(\log X)^2+O(X^* \log X)-(4\log X+O(1))\sum_{d\in D(X)}\sum_{n=1}^\infty \chi_d(n) F^\alpha_T(\log n)\frac{\Lambda(n)}{\sqrt{n}}\\
    &\qquad +4\sum_{d\in D(X)}\sum_{n_1,n_2}\chi_d(n_1n_2) F^\alpha_T(\log n_1)F^\alpha_T(\log n_2)\frac{\Lambda(n_1)\Lambda(n_2)}{\sqrt{n_1n_2}} \\
    &=X^*(\log X)^2\left(1-\frac{T}{\log X}\wh{F^\alpha}(0)+\frac{T^2}{(\log X)^2}2\sigma(F^\alpha)^2+o(1)\right),
\end{align*}
where we applied Lemma~\ref{lemma prime sum 2-level}. The two required conditions $\sqrt{\frac{\alpha+1}{\alpha-1}}< 2\beta$ and $\sqrt{\frac{\alpha+1}{\alpha-1}}(\alpha+1)<\pi^2\beta$ can be simultaneously satisfied if and only if $\beta>\frac{\pi}{2\sqrt{\pi^2-4}}=0.648\ldots$. After collecting the previous estimates we reach the following conclusion:

\begin{theorem}\label{theorem proportion quadratic}
Assume GRH. For $\beta\geq 0.649$,
\begin{equation}\label{equation proportion quadratic}
    \liminf_{X\to \infty}\frac{\#\mc{D}_\beta}{X^*}\geq \sup_{\substack{\sqrt{\frac{\alpha+1}{\alpha-1}}(\alpha+1)< \pi^2\beta \\ \sqrt{\frac{\alpha+1}{\alpha-1}}< 2\beta}}
    \dfrac{\left(1-\dfrac{\sqrt{\frac{\alpha+1}{\alpha-1}}(\alpha+1)}{\pi^2\beta}\right)^2}{1-\dfrac{2\sqrt{\frac{\alpha+1}{\alpha-1}}(\alpha+1)}{\pi^2\beta}+\dfrac{(6\alpha^2+\pi^2-3)(\alpha+1)}{24\pi^2\beta^2(\alpha-1)}}.
\end{equation}
\end{theorem}

Unfortunately we are not able to find a closed form expression for the right-hand side.

\begin{figure}[ht]
    \centering
    \includegraphics[width=7cm]{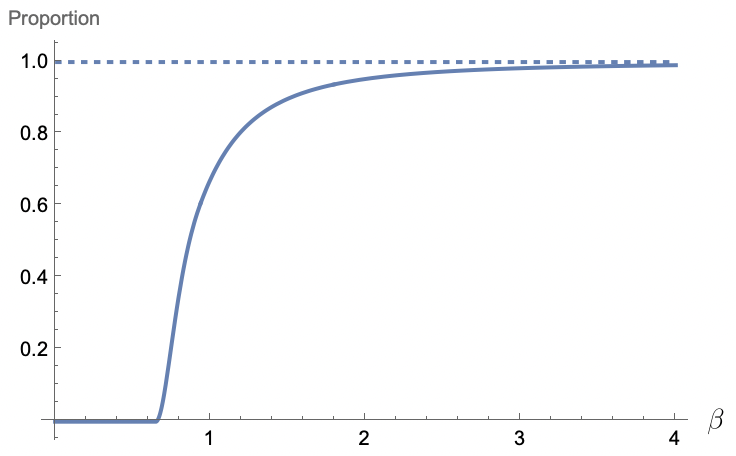}
    \caption{A plot of the right-hand side of \eqref{equation proportion quadratic} for small $\beta$.}
    \label{Figure_proportion (quadratic)}
\end{figure}

\section{Explicit estimates on \texorpdfstring{$|\gamma_\chi|$}{}}\label{section: explicit results}
It might be of interest to establish some explicit results for the low-lying zeros of $L(s,\chi)$. All the numerical computations in this section are carried out on \textit{Mathematica}. One natural question to ask is: for a given real number $t_0>0$, what is the least conductor $q_0$ such that $|\gamma_\chi|\leq t_0$ for all $q\geq q_0$? If $t_0>5/7$, the following estimate on the zero-counting function obtained by Bennett \textit{et al.} provides an unconditional upper bound on $q_0$:

\begin{theorem}[{\cite[Theorem 1.1]{BMOR}}]\label{explicit N(T,chi)}
    Let $\chi$ be a character with conductor $q>1$. If $t\geq 5/7$ and $\ell:=\log \frac{q(t+2)}{2\pi}>1.567$, then
    \[
    \left|N(t,\chi)-\left(\frac{t}{\pi}\log \frac{qt}{2\pi e}-\frac{\chi(-1)}{4}\right)\right|\leq 0.22737\ell+2\log(1+\ell)-0.5
    \]
    where $N(t,\chi)$ counts the number of zeros of $L(s,\chi)$ with $0<\beta<1$ and $|\gamma|\leq t$.
\end{theorem}

Take $t_0=1$ for example. As noted on \cite[pp. 1458]{BMOR}, it follows from Theorem~\ref{explicit N(T,chi)} that $N(1,\chi)\geq 1$, or $|\gamma_\chi|\leq 1$, whenever $q\geq 1.3\times 10^{47}$. Assuming GRH, we can substantially reduce this number by employing the positivity technique. Indeed, we find that for $\alpha=2.6$,
\[
I_\chi(F^\alpha_T)+2\sum_{n\leq e^T}F^\alpha_T(\log n)\frac{\Lambda(n)}{\sqrt{n}}\leq 20.98
\]
for all $T\leq T':=\sqrt{\frac{\alpha+1}{\alpha-1}}\pi$. (The choice of $\alpha$ is nearly optimal.) Thus $\sum_{\rho}\Phi(F^\alpha_T)(\rho)> 0$ for all $T\leq T'$ and $q\geq \pi e^{20.98}\approx 4.1\times 10^{9}$. As explained in \S\ref{section: preliminaries}, if $|\gamma_\chi|>1$, then the sum over zeros must be negative for $T=\sqrt{\frac{\alpha+1}{\alpha-1}}\pi/|\gamma_\chi|<T'$, from which we conclude that $|\gamma_\chi|\leq 1$ when $q\geq 4.1\times 10^{9}$. Similarly, for $t_0=5/7$, choosing $\alpha=2.9$ gives $|\gamma_\chi|\leq 5/7$ for $q\geq 2.1\times 10^{20}$.

Next we derive a fully explicit upper bound on $|\gamma_\chi|$ in terms of $q$ based on the proof of Theorem~\ref{theorem gamma_chi}.

\begin{theorem}\label{theorem: explicit gamma_chi}
    Assume GRH. For all $q\geq 10^{24}$,
    \[
    |\gamma_\chi|\leq \frac{\pi/2}{(\log\log q-1.43)}\sqrt{\frac{\log\log q}{\log\log q-2}}.
    \]
\end{theorem}
\begin{proof}
    We claim that for $\alpha\geq 3$, 
    \begin{equation}\label{explicit Lemma 5 and 6}
    \begin{split}
        &I_\chi(F^\alpha_T)\leq \min\left\{1.505,\frac{8.599}{T}\right\}\cdot (\alpha-1)+4.228,\\
        &2\sum_{n\leq e^T} F^\alpha_T(\log n)\frac{\Lambda(n)}{\sqrt{n}}\leq 4.156(\alpha-1)\frac{e^{T/2}}{T/2}-2.078(\alpha-1),
    \end{split}
    \end{equation}
    which are explicit versions of Lemmas~\ref{Lemma I_delta} and \ref{Lemma: sum over p}, respectively. Indeed, recall from the proof of Lemma~\ref{Lemma I_delta} that
    \[
    I_\chi(F^\alpha_T)\leq F^{\alpha \prime}(x^*) \frac{1}{2T}\int_0^{2T}\frac{xe^{-x/4}}{1-e^{-x}}\md x+\gamma_E+3\log 2+\frac{\pi}{2}
    \]
    for some $x^*\in (0,1)$. It is not hard to show that $|F^{\alpha \prime}(x)|\leq \alpha-1$ on $[0,1]$. Moreover, the integral converges to $17.197\ldots$ as $T\to \infty$, and for any $T>0$ the average of the function $\frac{xe^{-x/4}}{1-e^{-x}}$ on $[0,2T]$ is bounded by its global maximum $1.504\ldots$. This proves the first assertion. The second assertion follows from \eqref{prime sum in terms of psi} and the bound $\psi(x)<1.039x$ due to Rosser and Schoenfeld \cite[Theorem 12]{RS}.

    Now set $T=\sqrt{\frac{\alpha+1}{\alpha-1}}\pi/|\gamma_\chi|$ with $\alpha:=\log\log q-1\geq 3$ by our assumption $q\geq 10^{24}$. Then it can be deduced from \eqref{explicit Lemma 5 and 6} that
    \begin{align*}
        \frac{e^{T/2}}{T/2}\geq \frac{\log \frac{q}{\pi}-I_\chi(F^\alpha_T)+2.078(\alpha-1)}{4.156(\alpha-1)}\geq f(q):=\frac{\log q+0.573(\alpha-1)-5.373}{4.156(\alpha-1)},
    \end{align*}
    and consequently
    \[
    \frac{T}{2}>\log f(q)+\log\log f(q)+\log\left(1+\frac{\log\log f(q)}{\log f(q)}\right).
    \]
    For $\log\log q\leq 300$ we computationally verify that the above quantity is $>\log\log q-1.43$. For larger $q$, since $\frac{\log q}{4.156\log\log q}<f(q)<\log q$ and $\frac{\log \log x}{\log x}$ is decreasing for $x\geq 16$, we have
    \begin{align*}
        \frac{T}{2}>&\log\log q-\log 4.156 +\log\left(1-\frac{\log (4.156\log\log q)}{\log\log q}\right)+\log\left(1+\frac{\log\log \log q}{\log\log q}\right)\\
        >& \log\log q-\log 4.156-0.0053\\
        >&\log\log q-1.43,
    \end{align*}
    thereby completing the proof.
\end{proof}

\section{Extensions of Theorems~\ref{theorem gamma_chi}--\ref{theorem tilde gamma_chi} to general \texorpdfstring{$L$}{}-functions}\label{section: general L function}
We conclude by reformulating Theorems~\ref{theorem gamma_chi}, ~\ref{theorem n_chi} and ~\ref{theorem tilde gamma_chi} in the more general context. We impose the same assumptions on $L(s,\pi)$, the $L$-function under consideration, as in \cite[\S 4]{CCM1}. Suppose that it has an Euler product of the form
\[
L(s,\pi)=\prod_p \prod_{j=1}^m\left(1-\frac{\alpha_{j,\pi}(p)}{p^s}\right)^{-1},\quad \Re(s)>1
\]
where 
\[
|\alpha_{j,\pi}(p)|\leq p^\vartheta
\]
for some constant $0\leq \vartheta\leq 1$. Define the completed $L$-function by 
\[
\Lambda(s,\pi)=L_\infty(s,\pi)L(s,\pi) \qquad \text{where}\quad L_\infty(s,\pi)=N^{s/2}\prod_{j=1}^{m} \Gamma_\mathbb{R}(s+\mu_j)
\]
with $N>0$, $\Gamma_\mathbb{R}(s):=\pi^{-s/2}\Gamma(s/2)$, $\{\mu_j\}_{1\leq j\leq m}=\{\ov{\mu_j}\}_{1\leq j\leq m}$ and $\Re(\mu_j)>-1$, such that it satisfies the functional equation
\[
\Lambda(s,\pi)=\epsilon \ov{\Lambda(1-\ov{s},\pi)}, \qquad |\epsilon|=1.
\]
Suppose further that $\Lambda(s,\pi)$ has $r(\pi)$ poles at $s=0$ with $r(\pi)\leq m$, so that by the functional equation it has the same number of poles at $s=1$. The analytic conductor $C(\pi)$ of $\Lambda(s,\pi)$ is given by 
\[
C(\pi)=N\prod_{j=1}^m(|\mu_j|+3).
\]

Let $F$ be an admissible function. Mestre \cite[\S I.2]{Mestre} established the explicit formula
\begin{equation}\label{Mestre}
\begin{split}
    \sum_\rho \Phi(F)(\rho)=\log \frac{N}{\pi^m}+r(\pi)&\left(\Phi(F)(0)+\Phi(F)(1)\right)-\sum_{j=1}^m I_j(F)\\
    &-2\sum_{\substack{\text{$p$ prime}\\ 1\leq j\leq m,k\geq 1}}\Re(\alpha_{j,\pi}(p)^k)F(k\log p )\frac{\log p}{p^{k/2}}\\
    &-\sum_{-1<\Re(\mu_j)<-1/2}(\Phi(F)(-\mu_j)+\Phi(F)(1+\mu_j))\\
    &-\frac{1}{2}\sum_{\Re(\mu_j)=-1/2}(\Phi(F)(-\mu_j)+\Phi(F)(1+\mu_j))
\end{split}
\end{equation}
\footnote{The last two terms in \eqref{Mestre} do not appear in the original formula as found in \cite[\S I.2]{Mestre} because Mestre assumed that $\Re(\mu_j)\geq 0$, but only those $\mu_j$'s with $\Re(\mu_j)\leq -1/2$ show up when one moves the line of integration and applies the residue theorem. See also \cite[(4.6)]{CCM1} and \cite[Theorems 5.11 \& 5.12]{IwaKow}. }
where the sum runs over zeros of $\Lambda(s,\pi)$ with real part $0\leq \beta\leq 1$, $\Phi(F)$ was defined in \eqref{Phi(F)}, and
\[
I_j(F)=\int_0^\infty \left(\frac{F(x/2)e^{-(1/4+\Re(\mu_j)/2)x}}{1-e^{-x}}-\frac{e^{-x}}{x}\right)\md x.
\]

Let $|\gamma_\pi|$ (resp., $|\wt{\gamma_\pi}|$) and $n_\pi$ denote the height of the lowest (resp., lowest non-real) non-trivial zero of $\Lambda(s,\pi)$ and its order of vanishing at $s=\frac{1}{2}$, respectively. Assuming RH for $\Lambda(s,\pi)$, we apply \eqref{Mestre} to $F^\alpha_T$, $H_T$ and $G^\alpha_T$ defined in \eqref{def F^alpha}, \eqref{def H} and \eqref{def G}, respectively, and follow the same lines of argument as in \S\ref{section: preliminaries} and \S\ref{section: estimates for an individual L(s,chi)}. First, we can prove that
\[
\sum_{j=1}^m I_j(F^\alpha_T)\ll m\left(\frac{\alpha}{T}+1\right)e^{T/2},
\]
\[
\sum_{\substack{\text{$p$ prime}\\ 1\leq j\leq m,k\geq 1}}\Re(\alpha_{j,\pi}(p)^k)F^\alpha_T(k\log p )\frac{\log p}{p^{k/2}} \leq m\sum_{n\leq e^T} F^\alpha_T(\log n)\Lambda(n)n^{\vartheta-1/2}
\ll \frac{m \alpha e^{(1/2+\vartheta)T}}{T},
\]
and the last two terms in \eqref{Mestre} are both
\[
\ll m\int_{-T}^T F^\alpha_T(x)e^{x/2}\md x\ll m\alpha\int_{-T}^T \left(1-\frac{|x|}{T}\right)e^{x/2}\md x\ll \frac{m\alpha e^{T/2}}{T}.
\]
By modifying the proof of Theorem~\ref{theorem gamma_chi}, we find that
\[
|\gamma_\pi|\leq \left(\frac{1}{2}+\vartheta\right)\frac{\pi}{\log\log C(\pi)^{3/m}}+O\left(\frac{1}{\left(\log\log C(\pi)^{3/m}\right)^2}\right).
\]
Next, for $H_T$ we see that the second, third, and last two terms on the right-hand side of \eqref{Mestre} are all $\ll m\frac{e^{T/2}}{T}$, and the fourth term is $\ll m\frac{e^{(1/2+\vartheta)T}}{T}$. Hence
\[
n_\pi\leq \left(\frac{1}{2}+\vartheta\right)\frac{\log C(\pi)}{\log\log C(\pi)^{3/m}}+O\left(\frac{\log C(\pi)}{\left(\log\log C(\pi)^{3/m}\right)^2}\right).
\]
Lastly, similar estimates for $G^\alpha_T$ imply that
\[
|\wt{\gamma_\pi}|\leq \left(1+2\vartheta\right)\frac{\pi}{\log\log C(\pi)^{3/m}}+O\left(\frac{1}{\left(\log\log C(\pi)^{3/m}\right)^2}\right)
\]
in conjunction with the preceding bound on $n_\pi$.

\section{Acknowledgment}
I would like to thank Professor Ghaith Hiary for many fruitful discussions and invaluable suggestions. I would also like to thank the anonymous referee for their comments.

\printbibliography

\end{document}